\documentclass[a4paper]{amsart}

\usepackage{graphicx, amsmath, amssymb, amsthm}

\usepackage{natbib}

\usepackage{bm}

\newtheorem{thm}{Theorem}[section]
\newtheorem{cor}[thm]{Corollary}
\newtheorem{lem}[thm]{Lemma}
\newtheorem{prop}[thm]{Proposition}
\theoremstyle{definition}

\theoremstyle{remark}

\newtheorem{exam}[thm]{Example}
\numberwithin{equation}{section}

\newcommand{\R}{\mathbb{R}}

\newcommand{\norm}[1]{\left\Vert#1\right\Vert}
\newcommand{\abs}[1]{\left\vert#1\right\vert}
\newcommand{\set}[1]{\left\{#1\right\}}

\newcommand{\bfx}{\bm{x}}
\newcommand{\bfzero}{\bm{0}}
\newcommand{\bfinfty}{\bm{\infty}}

\newcommand{\bfe}{\bm{e}}

\newcommand{\bfs}{\bm{s}}

\newcommand{\bfU}{\bm{U}}
\newcommand{\bfu}{\bm{u}}

\newcommand{\bfX}{\bm{X}}

\newcommand{\bfy}{\bm{y}}
\newcommand{\bfZ}{\bm{Z}}

\begin{document}

\title[Conditional Distribution of Exceedance Counts]{Asymptotic Conditional Distribution of Exceedance Counts: Fragility Index with Different Margins}%
\author{Michael Falk\and Diana Tichy}%
\address{University of Wuerzburg, Institute of Mathematics, Emil-Fischer-Str. 30,
97074 W\"{u}rzburg, Germany}%
\email{falk@mathematik.uni-wuerzburg.de, d.tichy@mathematik.uni-wuerzburg.de}%


\begin{abstract}
Let $\bm X=(X_1,\dots,X_d)$ be a random vector, whose components
are not necessarily independent  nor are they required to have
identical distribution functions $F_1,\dots,F_d$. Denote by $N_s$
the number of exceedances among $X_1,\dots,X_d$ above a high
threshold $s$. The fragility index, defined by
$FI=\lim_{s\nearrow}E(N_s\mid N_s>0)$ if this limit exists,
measures the asymptotic stability of the stochastic system $\bm X$
as the threshold increases. The system is called stable if $FI=1$
and fragile otherwise. In this paper we show that the asymptotic
conditional distribution of exceedance counts (ACDEC)
$p_k=\lim_{s\nearrow}P(N_s=k\mid N_s>0)$, $1\le k\le d$, exists,
if the copula of $\bm X$ is in the domain of attraction of a
multivariate extreme value distribution, and if
$\lim_{s\nearrow}(1-F_i(s))/(1-F_\kappa(s))=\gamma_i\in[0,\infty)$
exists for $1\le i\le d$ and some
$\kappa\in\left\{1,\dots,d\right\}$. This enables the computation
of the FI corresponding to $\bm X$ and of the extended FI  as well
as of the asymptotic distribution of the exceedance cluster length
also in that case, where the components of $\bm X$ are not
identically distributed. 
\end{abstract}
\maketitle

\section{Introduction}\label{sec:introduction}

Let $\bfX=(X_1,\dots,X_d)$ be a random vector (rv), whose
components are  identically distributed but not necessarily
independent. The number of exceedances among $X_1,\dots,X_d$ above
the threshold $s$ is denoted by $N_s:=\sum_{i=1}^d
1_{(s,\infty)}(X_i)$. The fragility index (FI) corresponding to
$\bfX$ is the asymptotic conditional expected number of
exceedances, given that there is at least one exceedance, i.e., $
FI=\lim_{s\nearrow} E(N_s\mid N_s>0)$. The FI was introduced in
\cite{geluk07} to measure the stability of the stochastic system
$\set{X_1,\dots,X_d}$. The system is called \textit{stable} if $FI
=1$, otherwise it is called \textit{fragile}.

In the $2$-dimensional case, the FI is directly linked to the
upper  tail dependence coefficient $\lambda^{up}:=\lim_{t
\downarrow 0} P(X_2>F_2^{-1}(1-t)\mid  X_1>F_1^{-1}(1-t))$, which
goes back to \cite{gef58, gef59}  and \cite{sib60}. We have
$FI=2/(2-\lambda^{up})$, provided the df $F_1$, $F_2$ of $X_1$,
$X_2$ are continuous and $\lambda^{up}$ exists. In contrast to the
upper tail dependence coefficient, the FI presents a measure for
tail dependence in an arbitrary dimensions.

In \cite{falktichy10a} the asymptotic conditional  distribution
$p_k:=\lim_{s\nearrow}$ $ P(N_s=k\mid N_s>0)$ of the number of
exceedances was investigated, given that there is at least one
exceedance, $1\le k\le d$.

It turned out that this \textit{asymptotic conditional
distribution of exceedance counts} (ACDEC) exists, if the copula
$C$ corresponding to $\bfX$ is in the domain of attraction of a
(multivariate) extreme value distribution (EVD) $G$, denoted by
$C\in D(G)$, i.e. $
C^n\left(\left(1+\frac{x_1}n,\dots,1+\frac{x_d}n\right)\right)\to_{n\to\infty}G(\bfx),$
$\bfx\le\bfzero\in\R^d.$

In this paper we  investigate the ACDEC, dropping the assumption
that the margins    $X_i$, $1\le i\le d$, are identically
distributed. This will be done in Section \ref{sec:acdec}.
 If the ACDEC exists  then the FI exists and we have in
 particular $FI=\sum_{k=1}^d kp_k$.
In Section \ref{sec:fragility_index} we will compute the FI under
quite general conditions on $\bfX$.

The extended fragility index $FI(m)$ is the extension of the
$FI=FI(1)$ through the condition that there  are at least $m\ge 1$
exceedances, i.e.,
\[
FI(m)=\lim_{s\nearrow} E(N_s\mid N_s\ge m) =\frac{\sum_{k=m}^d kp_k}{\sum_{k=m}^d p_k},
\]
if the ACDEC exists. But now we encounter  the problem that the
denominator in the definition of $FI(m)$ may vanish: $\sum_{k=m}^d
p_k=0$. In Section \ref{sec:extended_fragility_index} we will
establish a characterization of $\sum_{k=m}^d p_k=0$ in terms of
tools from multivariate extreme value theory.

The total number of sequential time points at which a stochastic
process exceeds a high threshold is an \textit{exceedance cluster
length}. The asymptotic distribution as the threshold increases of
the remaining exceedance cluster length,  conditional on the
assumption that there is an exceedance at index
$\kappa\in\set{1,\dots,d}$, will be computed for
$\bfX=(X_1,\dots,X_d)$ in Section \ref{sec:sojourn_times}. It
turns out that this can be expressed in terms of the minimum of a
\textit{generator} of the $D$-norm, cf equation
\eqref{eqn:representation_of_D-norm_via_generator}.

\section{ACDEC}\label{sec:acdec}

By Sklar's Theorem (see, for example, \cite[Theorem
2.10.9]{nelsen06}) we can assume the representation
$(X_1,\dots,X_d)=(F_1^{-1}(U_1),\dots,F_d^{-1}(U_d))$, where $F_i$
is the (univariate) distribution function (df) of $X_i$, $1\le
i\le d$, and the rv $\bfU=(U_1,\dots,U_d)$ follows a copula on
$\R^d$, i.e., each $U_i$ is uniformly on $(0,1)$ distributed,
$1\le i\le d$. By $F^{-1}(q):=\inf\set{t\in\R:\,F(t)\ge q}$,
$q\in(0,1)$, we denote the generalized inverse of a df $F$.

The following condition is crucial for the present paper. It substitutes the condition of equal margins $F_1=\dots=F_d$ in \cite{falktichy10a}. By $\omega(F):=\sup\{ F^{-1}(q):\,q\in(0,1)\}$ $=\sup\set{t\in\R:\,F(t)<1}$ we denote the upper endpoint of a df $F$.

We require throughout the existence of an index $\kappa\in\set{1,\dots,d}$ with $\omega(F_\kappa)=:\omega^*$, such that
\begin{equation*}
\lim_{s\uparrow \omega^*}\frac{1-F_i(s)}{1-F_\kappa(s)}=\gamma_i\in[0,\infty),\qquad 1\le i\le d.\label{cond:crucial_condition_on_tails}\tag{C}
\end{equation*}
Note that condition \eqref{cond:crucial_condition_on_tails} implies $\omega(F_i)\le \omega^*$ for each $i$, since otherwise we would get $\gamma_i=\infty$, which is excluded. We, thus, have $\omega^*=\max_{i\le d} \omega(F_i)$.

The following result is taken from \cite{aulbayfalk09}. By $\bfe_i$ we denote the $i$-th unit vector in $\R^d$, $1\le i\le d$; all operations on vectors such as $\bfx\le\bfzero\in\R^d$ are meant componentwise.

\begin{prop}\label{prop:equivalences_of_domain}
An arbitrary copula $C$ on $\R^d$ is in the domain of attraction of an EVD $G$ if and only if there exists a norm $\norm{\cdot}_D$ on $\R^d$ with $\norm{\bfe_i}_D=1$, $1\le i\le d$, such that
\begin{equation*}\label{cond:equivalences_of_domain}
C(\bfy)=1-\norm{\bfy-\bm{1}}_D+
o\left(\norm{\bfy-\bm{1}}_D\right),
\end{equation*}
uniformly for $\bfy\in[0,1]^d$. In this case
$G(\bfx)=\exp\left(-\norm{\bfx}_D\right)$, $\bfx\le\bfzero\in\R^d$.
\end{prop}

The following result is an immediate consequence of Proposition \ref{prop:equivalences_of_domain} and the equivalence $F^{-1}(q)\le t \iff q\le F(t)$, $q\in(0,1)$, $t\in\R$, which holds for an arbitrary df $F$.

\begin{cor}\label{cor:expansion_of_P(X_k_le s)}
Suppose that the copula $C$ corresponding to the rv $\bfX$ is in the domain of attraction of an EVD $G$ and that condition \eqref{cond:crucial_condition_on_tails} is satisfied. Then there exists a norm $\norm\cdot_D$ on $\R^d$ with $\norm{\bfe_i}_D=1$, $1\le i\le d$, such that for any nonempty index set $K\subset\set{1,\dots,d}$
\[
P(X_k\le s,\,k\in K)=1-(1-F_\kappa(s))\norm{\sum_{k\in K}\gamma_k\bfe_k}_D+o(1-F_\kappa(s))
\]
as $s\uparrow\omega^*$.
\end{cor}

The following result provides the asymptotic unconditional distribution of exceedance counts.

\begin{lem}\label{lem:unconditional_acdec}
Under the conditions of Corollary \ref{cor:expansion_of_P(X_k_le s)} we obtain with $c:=1-F_\kappa(s)$
\begin{align*}
a_k&:=\lim_{s\uparrow\omega^*}\frac{P(N_s=k)}c\\
&=\sum_{0\le j\le k}(-1)^{k-j+1}\binom{d-j}{k-j}\sum_{\emptyset\not=T\subset\set{1,\dots,d} \atop \abs{T}=d-j}\norm{\sum_{i\in T}\gamma_i\bfe_i}_D
\end{align*}
for $1\le k\le d$, and
\[
a_0:=\lim_{s\uparrow\omega^*}\frac{1-P(N_s=0)}c= \norm{\sum_{j=1}^d\gamma_j\bfe_j}_D.
\]
\end{lem}

\begin{proof}
Corollary \ref{cor:expansion_of_P(X_k_le s)} implies
\[
P(N_s=0)=1-c\norm{\sum_{j=1}^d\gamma_j\bfe_j}_D+o(c),
\]
for $s\uparrow \omega^*$.

From the additivity formula, Corollary \ref{cor:expansion_of_P(X_k_le s)} and the equality $\sum_{\emptyset\not= T\subset S}(-1)^{\abs{T}+1}=1$ for any nonempty subset $S\subset\set{1,\dots,d}$, we obtain for $1\le k\le d$ as $s\uparrow\omega^*$\allowdisplaybreaks[4]
\begin{align*}
&P(N_s=k)\\
&=\sum_{S\subset\set{1,\dots,d}\atop \abs{S}=k} P\left(X_i>s,\,i\in S,\,X_j\le s,\,j\in S^\complement\right)\\
&=\sum_{S\subset\set{1,\dots,d}\atop \abs{S}=k} P\left(X_i>s,\,i\in S \mid X_j\le s,\,j\in S^\complement\right) P\left(X_j\le s,\,j\in S^\complement\right)\\
&=\sum_{S\subset\set{1,\dots,d}\atop \abs{S}=k} \left(1-\sum_{\emptyset\not= T\subset S}(-1)^{\abs{T}+1} P\left(X_i\le s,\,i\in T \mid X_j\le s,\,j\in S^\complement\right)\right)\\
&\hspace*{5cm}\times P\left(X_j\le s,\,j\in S^\complement\right)\\
&=\sum_{S\subset\set{1,\dots,d}\atop \abs{S}=k}\left(P\left(X_j\le s,\,j\in S^\complement\right)- \sum_{\emptyset\not= T\subset S}(-1)^{\abs{T}+1} P\left(X_i\le s,\,i\in T \cup S^\complement\right)\right)\\
&=\sum_{S\subset\set{1,\dots,d}\atop \abs{S}=k}\left(1-c\norm{\sum_{j\in S^\complement}\gamma_j\bfe_j}_D - \sum_{\emptyset\not= T\subset S}(-1)^{\abs{T}+1} \left(1-c \norm{\sum_{j\in T\cup S^\complement}\gamma_j\bfe_j}_D\right)\right)\\
&\hspace*{7cm}+ o(c) \\
&= c \sum_{S\subset\set{1,\dots,d}\atop \abs{S}=k} \sum_{T\subset S}(-1)^{\abs{T}+1} \norm{\sum_{j\in T\cup S^\complement}\gamma_j\bfe_j}_D + o(c).
\end{align*}

With a suitable index transformation we get
\begin{align*}
P(N_s=k)&=c \sum_{S\subset\set{1,\dots,d}\atop \abs{S}=k}\sum_{0\le r\le \abs S}(-1)^{r+1}\sum_{K\subset S\atop \abs K=r} \norm{\sum_{i\in K\cup S^\complement:=T\atop \abs T=r+d-k}\gamma_i\bfe_i}_D + o(c)\\
&=c \sum_{0\le j \le k}(-1)^{k-j+1} \binom{d-j}{k-j}
\sum_{T\subset\set{1,\dots,d}\atop \abs T=d-j} \norm{\sum_{j\in
T}\gamma_j\bfe_j}_D + o(c),
\end{align*}
which completes the proof of Lemma \ref{lem:unconditional_acdec}.
\end{proof}

Note that $a_0>0$ as $\gamma_k=1$ and that $a_k\ge 0$, $1\le k\le d$, in Lemma \ref{lem:unconditional_acdec}. The following main result of this section is, therefore, an immediate consequence of Lemma \ref{lem:unconditional_acdec}. It provides the ACDEC also in the case, where the components $X_i$ of the rv $\bfX=(X_1,\dots,X_d)$ are not identically distributed.

\begin{thm}[ACDEC]\label{th:acdec}
Under the conditions of Corollary \ref{cor:expansion_of_P(X_k_le s)} we have that the limits
\[
p_k:=\lim_{s\uparrow \omega^*}P(N_s=k\mid N_s>0)=\frac{a_k}{a_0},\qquad 1\le k\le d,
\]
exist and that they define a probability distribution on $\set{1,\dots,d}$.
\end{thm}

For the usual $\lambda$-norm $\norm{\bfx}_\lambda=\left(\sum_{1\le
i\le d}\abs{x_i}^\lambda\right)^{1/\lambda}$,  $\bfx\in\R^d$,
$\lambda\in[1,\infty)$, we obtain, for example,
$a_0=\left(\sum_{1\le i\le d}\gamma_i^\lambda\right)^{1/\lambda}$
and
\[a_k=\sum_{0\le j\le k} (-1)^{k-j+1}
\binom{d-j}{k-j}\sum_{\emptyset\not=T\subset\set{1,\dots,d}\atop
\abs{T}=d-j} \left(\sum_{i\in
T}\gamma_i^\lambda\right)^{1/\lambda},\qquad 2\le k\le d.
\]
For $\lambda=1$, which is the case of independent margins of $G$,
we obtain in particular $a_0=\sum_{1\le i\le d}\gamma_i=a_1$,
$a_k=0$, $2\le k\le d$, and, thus, $p_1=1$, $p_k=0$, $2\le k\le
d$.

\section{The Fragility Index}\label{sec:fragility_index}

The following theorem is the main result of this section.

\begin{thm}\label{th:fragility_index}
Under the conditions of Corollary \ref{cor:expansion_of_P(X_k_le s)} we have
\[
FI=\frac{\sum_{i=1}^d\gamma_i}{\norm{\sum_{i=1}^d\gamma_i\bfe_i}_D}\in[1,d].
\]
\end{thm}

\begin{proof}
We have \allowdisplaybreaks[4]
\begin{align*}
E(N_s\mid N_s>0)&=\sum_{i=1}^d E\left(1_{(s,\infty)}(X_i) \mid N_s>0\right)\\
&= \sum_{i=1}^d \frac{P(X_i>s)}{1-P(N_s=0)}\\
&= \sum_{i=1}^d \frac{1-F_i(s)}{1-F_\kappa(s)}\, \frac{1-F_\kappa(s)}{1-P(N_s=0)}\\
&\to_{s\to\infty} \frac{\sum_{i=1}^d\gamma_i}{\norm{\sum_{i=1}^d\gamma_i\bfe_i}_D}.
\end{align*}
by Lemma \ref{lem:unconditional_acdec} and condition \eqref{cond:crucial_condition_on_tails}.
\end{proof}

It is well known that an arbitrary $D$-norm satisfies the inequality $\norm{\bfx}_\infty\le \norm{\bfx}_D\le \norm{\bfx}_1$, $\bfx\ge\bfzero\in\R^d$; see, for example \cite[(4.37)]{fahure10}. The range of the FI in Theorem \ref{th:fragility_index} is, consequently, $[1,d]$.

Suppose that $\gamma_i>0$, $1\le i\le d$. Then it follows from \cite{ta88} that
\[
\norm{\sum_{i=1}^d\gamma_i\bfe_i}_D=\sum_{i=1}^d\gamma_i \iff\norm{\cdot}_D=\norm{\cdot}_1,
\]
where $\norm{\cdot}_D=\norm{\cdot}_1$ is the case of independence of the margins of $G$. We, thus, obtain in case $\gamma_i>0$, $1\le i\le d$,
\[
FI=1\iff\norm{\cdot}_D=\norm{\cdot}_1\iff \mbox{ independence of the margins of }G.
\]

In case of complete dependence of $G$, i.e., if $\norm{\bfx}_D
=\norm{\bfx}_\infty=\max_{1\le i\le d}\abs{x_i}$, we obtain for
general $\gamma_i\ge 0$ that $FI=\sum_{i=1}^d\gamma_i/\max_{1\le
i\le d}\gamma_i$.

\begin{exam}[Weighted Pareto]
Let $Y_1,\dots,Y_m$ be independent and identically Pareto
distributed rv with parameter $\alpha>0$. Put
$X_i:=\sum_{j=1}^m\lambda_{ij}Y_j$, $1\le i\le d,$ where the
weights $\lambda_{ij}$ are nonnegative and satisfy
$\sum_{j=1}^m\lambda_{ij}^\alpha=1$, $1\le i\le d$.

The df of the rv $\bfX=(X_1,\dots,X_d)$ is in the domain of attraction of the EVD
\[
G^*(\bfs)=\exp\left(-\sum_{j=1}^m\max_{i\le d}\left(\frac{\lambda_{ij}}{s_i}\right)^\alpha\right),\qquad \bfs=(s_1,\dots,s_d)>\bfzero,
\]
with standard Fr\'{e}chet margins
$G_k(\bfs)=\exp\left(-s^{-\alpha}\right)$, $s>0$, $1\le k\le d$.
This can be seen by proving that for $\bfs>\bfzero\in\R^d$
\[
P\left(\sum_{j=1}^m\lambda_{ij}Y_j\le n^{1/\alpha} s_i,\,1\le i\le
d\right) =1-\frac 1n\left(\sum_{j=1}^m\max_{i\le
d}\left(\frac{\lambda_{ij}}{s_i}\right)^\alpha+o(1)\right),
\]
which follows from tedious but elementary computations, using conditioning on $Y_j=y_j$, $j=2,\dots,m$.

We, thus, obtain that the copula pertaining to $\bfX$ is
in  the domain of attraction of
$G(\bfx)=\exp\left(-\norm{\bfx}_D\right)$,
$\bfx\le\bfzero\in\R^d$, where $\norm{\bfx}_D:=\sum_{j=1}^m
\left(\max_{i\le d}\left( \lambda_{ij}
^\alpha\abs{x_i}\right)\right)$, $\bfx\in\R^d$.

From \cite[Lemma A 3.26]{emkm97} we obtain that the df $F_i$ of
$X_i$ satisfies $1-F_i(s)\sim
s^{-\alpha}\sum_{j=1}^m\lambda_{ij}^\alpha=s^{-\alpha}$, $1\le
i\le d$, as $s\to\infty$ and, thus,
\[
\gamma_i=\lim_{s\to\infty}\frac{1-F_i(s)}{1-F_\kappa(s)}=1,\qquad 1\le i\le d,
\]
where $\kappa\in\set{1,\dots,d}$ can be chosen arbitrarily. As a consequence we obtain for the fragility index
\[
FI=\frac{\sum_{i=1}^d\gamma_i}{\norm{\sum_{i=1}^d \gamma_i\bfe_i}_D}
=\frac d{\sum_{j=1}^m \max_{i\le d}\lambda_{ij}^\alpha}.
\]
\end{exam}

\begin{exam}[GPD-Copula]\label{exam:gpd-copula}
Take an arbitrary rv $\bfZ$ that realizes in $[0,c]^d$ and which
satisfies $E(Z_i)=1$, $1\le i\le d$. Choose
$\beta_1,\dots,\beta_d>0$ and let $U$ be a rv, which is uniformly
on $(0,1)$ distributed and that is independent of $\bfZ$. Put
$\bfX:=(\beta_1Z_1,\dots,\beta_dZ_d)/U$. Then $F_i(x)=P(X_i\le
x)=1-\frac{\beta_i} x$, $x\ge c\beta_i$, $1\le i\le d$, and the
copula of $\bfX$ is in the domain of attraction of the EVD
$G(\bfx)=\exp(-\norm{\bfx}_D)$, $\bfx\le \bfzero\in\R^d$, with
$\norm{\bfx}_D=E\left(\max_{1\le i\le d}(\abs{x_i}Z_i)\right)$,
$\bfx\in\R^d$.

Let $\beta_\kappa=\max_{1\le i\le d}\beta_i$. Then we have
\[
\frac{1-F_i(s)}{1-F_\kappa(s)}=\frac{\beta_i}{\beta_\kappa}=:\gamma_i, \qquad s\ge c\beta_\kappa,\;1\le i\le d,
\]
and we obtain for the fragility index corresponding to $\bfX$
\[
FI=\frac{\sum_{i=1}^d\gamma_i}{E\left(\max_{1\le i\le d}\gamma_iZ_i\right)}.
\]
Note that the copula $C$ of $\bfX$ is actually a \textit{GPD
copula} ((multivariate)  generalized Pareto distribution),
characterized by the equation $C(\bfu)=1-\norm{\bm 1-\bfu}_D$ for
$\bfu\in[0,1]^d$ close to $\bm 1$, see \cite{aulbayfalk09}. If
$Z_1=\dots=Z_d$ a.s., then we obtain the maximum-norm
$\norm{\bfx}_D=\max_{1\le i\le d}\abs{x_i}$, and
$FI=\sum_{i=1}^d\gamma_i/\max_{1\le i\le d}\gamma_i$.
\end{exam}

\section{The Extended Fragility Index}\label{sec:extended_fragility_index}

The extended FI is the asymptotic expected number of exceedances
above a high threshold, conditional on the assumption that there
are at least $m\ge 1$ exceedances:
\[
FI(m):=\lim_{s\nearrow} E(N_s\mid N_s\ge m),\qquad 1\le m\le d.
\]
If the ACDEC corresponding to $X_1,\dots X_d$ exists, then, obviously,
\begin{equation}\label{eqn:extended_fragility_index_via_acdec}
FI(m)=\frac{\sum_{k=m}^d k p_k}{\sum_{k=m}^d p_k},\qquad 1\le m\le d.
\end{equation}
But now we encounter the problem that we might divide by $0$ in \eqref{eqn:extended_fragility_index_via_acdec}, i.e., $\sum_{k=m}^d p_k$ can vanish if $m\ge 2$. This is, for example, true for the $L_1$-norm. But there are other norms in dimension $d\ge 3$ such that $\sum_{k=m}^d p_k=0$, see \cite{falktichy10a}. In this section we establish a characterization of $\sum_{k=m}^d p_k=0$ also in that case, where the initial $X_1,\dots,X_d$ follow different distributions.

\begin{lem}
Assume the conditions of Corollary \ref{cor:expansion_of_P(X_k_le s)} and put $I:=
\{i\in\{1,\dots,d\}:\,\gamma_i=0\}$. Then we obtain $\sum_{k=m}^dp_k=0$ for $m>m^*:=\abs{I^\complement}=d-\abs I$.
\end{lem}

\begin{proof}
Without loss of generality we can assume that $I\not=\emptyset$. Recall, moreover, that $\gamma_\kappa=1$, i.e., $I\not=\set{1,\dots,d}$ as well. We have
\[
a_k=\lim_{s\uparrow\omega^*}\frac{P(N_s=k)}{1-F_\kappa(s)}
=\lim_{s\uparrow\omega^*}\sum_{S\subset\set{1,\dots,d}\atop\abs
S=k} \frac{P\left(X_i>s,\,i\in S,\,X_j\le s,\,j\in
S^\complement\right)}{1-F_\kappa(s)}.
\]
If $\abs S=k\ge m^*+1$, then $S$ must contain an index $i_S$, say, with $i_S\in I$. We, thus, obtain for $k\ge m^*+1$
\[
a_k\le \limsup_{s\uparrow\omega^*}
\sum_{S\subset\set{1,\dots,d}\atop\abs S=k}
\frac{P(X_{i_S}>s)}{1-F_\kappa(s)} =
\sum_{S\subset\set{1,\dots,d}\atop\abs S=k}
\lim_{s\uparrow\omega^*} \frac{1-F_{i_S}(s)}{1-F_\kappa(s)}=0.
\]
\end{proof}

The following characterization is the main result of this section. It is formulated in terms of different representations of a multivariate EVD $G$ on $\R^d$ with standard negative exponential margins $G(x\bfe_i)=\exp(x)$, $x\le 0$, $1\le i\le d$. We have for $\bfx\le \bfzero\in\R^d$
\begin{align*}
G(\bfx)&=\exp\left(-\norm{\bfx}_D\right)\tag{Hofmann}\\
 &=\exp\left(-\int_{S_d}\max(-u_ix_i)\,\mu(d\bfu)\right)\tag{Pickands-de Haan-Resnick} \\
 &=\exp\left(-\nu\left([-\bfinfty,\bfx]^\complement\right)\right), \tag{Balkema-Resnick}
\end{align*}
where $\norm\cdot_D$ is some norm on $\R^d$ with $\norm{\bfe_i}_D=1$, $1\le i\le d$, $\mu$ is the \emph{angular measure} on the unit simplex
$S_d=\big\{\bfu\in[0,1]^d:\,\sum_{i\le d}u_i=1\big\}$, satisfying $\mu(S_d)=d$, $\int_{S_d}u_i\,\mu(d\bfu)=1$, $1\le i\le d$, and $\nu$ is the $\sigma$-finite \emph{exponent measure} on
$[-\infty,0]^d\backslash\set{\bfinfty}$; for details we refer to \cite{fahure10}. We also include the fact that each $D$-norm can be generated by nonnegative and bounded rv $Z_1,\dots,Z_d$ with $E(Z_i)=1$, $1\le i\le d$, as
\begin{equation}\label{eqn:representation_of_D-norm_via_generator}
\norm{\bfx}_D=E\left(\max_{1\le i\le d}(\abs{x_i}Z_i)\right),\qquad \bfx=(x_1,\dots,x_d)\in\R^d.
\end{equation}
This is a consequence of the Pickands-de Haan-Resnick
representation. The rv $\bfZ=(Z_1,\dots,Z_d)$  is called
\textit{generator} of $\norm\cdot_D$. Note that each rv
$\bfZ=(Z_1,\dots,Z_d)$ of nonnegative and bounded rv $Z_i$ with
$E(Z_i)=1$ generates a $D$-norm via equation
\eqref{eqn:representation_of_D-norm_via_generator}.

\begin{prop}
Assume the conditions of Corollary \ref{cor:expansion_of_P(X_k_le s)} and put $I=\{i\in\{1,\dots,d\}:\,\gamma_i=0\}$. Then we have $\sum_{k=m}^dp_k=0$ for some $m\le m^*=\abs{I^\complement}$ if and only if we have for each subset $K\subset I^\complement$ with at least $m$ elements\allowdisplaybreaks[4]
\begin{align}
&\lim_{s\uparrow \omega^*} \frac{P(X_k>s,\,k\in K)}{1-F_\kappa(s)}=0\label{prop:characterization_for_fi(m) via_sets}\\
&\iff \sum_{T\subset K}(-1)^{\abs{T}-1}\norm{\sum_{i\in T}x_i\bfe_i}_D=0\quad \mbox{for all } \bfx\ge\bfzero\in\R^d \nonumber\\
&\iff  \sum_{T\subset K}(-1)^{\abs{T}-1}\norm{\sum_{i\in T}\bfe_i}_D=0 \nonumber\\
&\iff  \min_{k\in K}Z_k=0\quad a.s.\label{prop:characterization_for_fi(m)_via_generator}\\
&\iff \mu\left(\set{\bfu\in S_d:\min_{i\in K}u_i>0}\right)=0 \nonumber\\
&\iff \nu\left(\times_{k\in K}(-\infty,0]  \times_{i\not\in K}[-\infty,0]\right)=0, \nonumber
\end{align}
 i.e., the projection $\nu_{K}:=\nu*(\pi_i,\,i\in K)$ of the exponent measure $\nu$ onto its components $i\in K$ is the null measure on $(-\infty,0]^{\abs K}$.
\end{prop}

While in the (bivariate) case $K=\set{k_1,k_2}$ the condition
\begin{align*}
&\sum_{T\subset K}(-1)^{\abs{T}-1}\norm{\sum_{i\in T}\bfe_i}_D=0\\
&\iff \norm{\bfe_{k_1}}_D + \norm{\bfe_{k_2}}_D - \norm{\bfe_{k_1}+\bfe_{k_2}}_D = 0\\
&\iff \norm{\bfe_{k_1}+\bfe_{k_2}}_D = 2 =\norm{\bfe_{k_1}+\bfe_{k_2}}_1
\end{align*}
implies by Takahashi's Theorem (\cite{ta88}) independence  of the
marginal distributions $k_1,k_2$ of the EVD
$G(\bfx)=\exp(-\norm{\bfx}_D)$, $\bfx\le \bfzero\in\R^d$, this is
no longer true for $\abs K\ge 3$. Take, for example, a rv $\xi$
that attains only the values 1;2;3 with probability 1/6; 1/3; 1/2
and put
\begin{equation*}
Z_1:=
\begin{cases}
0&\mbox{ if }\xi=1\\
\frac 65& \mbox{ elsewhere}
\end{cases},\quad
  Z_2:=\begin{cases}
0&\mbox{ if }\xi=2\\
\frac 32& \mbox{ elsewhere}
\end{cases}, \quad
Z_3:=\begin{cases}
0&\mbox{ if }\xi=3\\
 2& \mbox{ elsewhere}
\end{cases}.
\end{equation*}
Then $E(Z_i)=1$, $i=1,2,3$, $\min_{1\le i\le 3}Z_i=0$, $E(\max_{1\le i\le 3}Z_i)<3$ as well as $E(\max(Z_i,Z_j))<2$ for all $1\le i\not=j\le 3$, i.e., there is no marginal independence among $Z_1,Z_2,Z_3$.

\begin{proof}
We have by Theorem \ref{th:acdec} and Lemma \ref{lem:unconditional_acdec}
\begin{align*}
&\sum_{k=m}^d p_k=0\\
&\iff \lim_{s\uparrow\omega^*}\frac{P(N_s\ge m)}{1-F_\kappa(s)}=0\\
&\iff \lim_{s\uparrow\omega^*}\frac{P\left(\bigcup_{K\subset\set{1,\dots,d}\atop \abs K\ge m}\set{X_k>s,\,k\in K}\right)} {1-F_\kappa(s)}=0\\
&\iff \lim_{s\uparrow\omega^*}\frac{P(X_k>s,\,k\in K)} {1-F_\kappa(s)}=0\mbox{ for any }K\subset \set{1,\dots,d}\mbox{ with }\abs K\ge m\\
&\iff \lim_{s\uparrow\omega^*}\frac{P(X_k>s,\,k\in K)} {1-F_\kappa(s)}=0\mbox{ for any }K\subset I^\complement\mbox{ with }\abs K\ge m ,
\end{align*}
which is equivalence \eqref{prop:characterization_for_fi(m) via_sets}. Note that $\sum_{T\subset K}(-1)^{\abs T-1}\max_{i\in T}a_i=\min_{k\in K} a_k$ for any set $\set{a_k:\,k\in K}$ of real numbers, which can be seen by induction. We, consequently, have
\begin{equation*}
\sum_{T\subset K} (-1)^{\abs T-1}\norm{\sum_{i\in T}\bfe_i}_D
 = \sum_{T\subset K} (-1)^{\abs T-1}E\left(\max_{i\in T}Z_i\right)
 = E\left(\min_{i\in T}Z_i\right)
\end{equation*}
and, thus,
\[
\sum_{T\subset K} (-1)^{\abs T-1}\norm{\sum_{i\in T}\bfe_i}_D =0
\iff E\left(\min_{i\in T}Z_i\right)=0 \iff \min_{k\in K}Z_k=0 \;
a.s.
\]
The other equivalences follow from Proposition 5.2 in \cite{falktichy10a}.
\end{proof}

\section{Exceedance Cluster Lengths}\label{sec:sojourn_times}
The total number of sequential time points at which a stochastic process exceeds a high threshold is an exceedance cluster length.
The mathematical tools developed in the preceding sections enable the computation of its distribution as well. Precisely, denote by $L_\kappa(s)$ the number of sequential exceedances above the threshold $s$, if we have an exceedance at $\kappa\in\set{1,\dots,d}$, i.e.
\[
L_\kappa(s):=\sum_{k=0}^{d-\kappa} k 1\left(X_\kappa>s,\dots,X_{\kappa+k}>s, X_{\kappa+k+1}\le s\right).
\]
We have, in particular, $L_d(s)=0=L_\kappa(s)$, if
$X_{\kappa+1}\le s$. We suppose throughout this section that
condition \eqref{cond:crucial_condition_on_tails}  holds for the
index $\kappa\in\set{1,\dots,d}$. The following auxiliary result
will be crucial.

\begin{lem}\label{lem:sojourn_time}
Assume the conditions of Corollary \ref{cor:expansion_of_P(X_k_le s)}. Then we obtain for $\kappa\in\set{1,\dots,d}$ as $s\nearrow \omega^*$
\begin{align*}
P\left(L_\kappa(s)\ge k\mid X_\kappa>s\right) &= P\left(X_\kappa>s,\dots,X_{\kappa+k}>s \mid  X_\kappa>s\right)\\
&= \sum_{\emptyset\not=T\subset\set{\kappa,\dots,\kappa+k}} (-1)^{\abs T+1}\norm{\sum_{i\in T}\gamma_i\bfe_i}_D + o(1)\\
&=:s_\kappa(k)+o(1),\qquad 0\le k\le d-\kappa.
\end{align*}
\end{lem}

\begin{proof}
From the additivity formula we obtain
\begin{align*}
&P\left(X_\kappa>s,\dots,X_{\kappa+k}>s \mid  X_\kappa>s\right)\\
&=\frac{1-P\left(\bigcup_{0\le i\le k}\set{X_{\kappa+i}\le s}\right)} {1-F_\kappa(s)}\\
&=\frac{1-\sum_{\emptyset\not=T\subset\set{\kappa,\dots,\kappa+k}}(-1)^{\abs T+1}P\left(X_i\le s,\, i \in T\right)}{1-F_\kappa(s)}\\
&=\frac{1-\sum_{\emptyset\not=T\subset\set{\kappa,\dots,\kappa+k}}(-1)^{\abs T+1} \left(1-c\norm{\sum_{i\in T}\gamma_i\bfe_i}_D\right)+ o(1-F_\kappa(s))} {1-F_\kappa(s)}\\
&=\sum_{\emptyset\not=T\subset\set{\kappa,\dots,\kappa+k}}(-1)^{\abs T+1} \norm{\sum_{i\in T}\gamma_i\bfe_i}_D + o(1).
\end{align*}
\end{proof}

\begin{cor}
Suppose in addition to the assumptions in Corollary \ref{cor:expansion_of_P(X_k_le s)} that $\bfZ$ is a generator of the $D$-norm $\norm{\cdot}_D$. Then we obtain for $\kappa\in\set{1,\dots,d}$ as $s\nearrow \omega^*$
\[
P(X_\kappa > s,\dots,X_{\kappa+k}>s\mid X_\kappa > s)= E\left(\min_{\kappa\le i\le \kappa+k} (\gamma_i Z_i)\right) + o(1),
\]
for $0\le k\le d-\kappa$.
\end{cor}

Though the distribution of a generator of a $D$-norm is not uniquely determined, the preceding result entails that the numbers $E\left(\min_{\kappa\le i\le \kappa+k} (\gamma_i Z_i)\right)$, $0\le k\le d-\kappa$, are uniquely determined by the $D$-norm.

The asymptotic distribution of the exceedance cluster length, conditional on the assumption that there is an exceedance at time point $\kappa\in\set{1,\dots,d}$, is an immediate consequence of Lemma \ref{lem:sojourn_time}. It follows from the equation
\[
P(L_\kappa(s)=k\mid X_\kappa > s)=P(L_\kappa(s)\ge k\mid X_\kappa
> s)- P(L_\kappa(s)\ge k+1\mid X_\kappa > s).
\]
Note, moreover, that $P(L_\kappa(s)=0\mid X_\kappa > s)=1$ for $\kappa=d$.

\begin{prop}\label{prop:expansion_of_exceedance_cluster_length}
Assume the conditions of Corollary \ref{cor:expansion_of_P(X_k_le s)}. Then we have for $\kappa < d$ as $s\nearrow\omega^*$
\begin{align*}
&P(L_\kappa(s)=k\mid X_\kappa > s)\\
&=\begin{cases}
\sum_{\emptyset\not= T\subset\set{\kappa,\dots,d}} (-1)^{\abs T+1}\norm{\sum_{i\in T}\gamma_i\bfe_i}_D + o(1), \\
\hspace*{7cm}k=d-\kappa,\\
\sum_{T\subset\set{\kappa,\dots,\kappa+k}}(-1)^{\abs T+1}\norm{\gamma_{\kappa+k+1}\bfe_{\kappa+k+1}+ \sum_{i\in T}\gamma_i\bfe_i}_D + o(1),\\
\hspace*{7cm} 0\le k< d-\kappa.
\end{cases}
\end{align*}
\end{prop}

We obtain, for example, for $\kappa < d$
\[
P(L_\kappa(s)=0\mid X_\kappa > s) = \norm{\bfe_\kappa+\gamma_{\kappa+1}\bfe_{\kappa+1}}_D - 1 + o(1),
\]
which converges to $\gamma_{\kappa+1}$ if $\norm\cdot_D=\norm\cdot_1$. Recall that $\gamma_\kappa=1$.

In terms of a generator $\bfZ$ of a $D$-norm, Proposition \ref{prop:expansion_of_exceedance_cluster_length} becomes the following result.

\begin{cor}
Assume in addition to the conditions of Corollary \ref{cor:expansion_of_P(X_k_le s)} that $\bfZ$ is a generator of the $D$-norm $\norm{\cdot}_D$. Then we have for $\kappa < d$ as $s\nearrow\omega^*$
\begin{itemize}
\item[(i)] $P(L_\kappa(s)= k\mid X_\kappa > s)$
\[
=\begin{cases}
E\left(\min_{\kappa\le i\le d}(\gamma_i Z_i)\right)+o(1),&k=d-\kappa\\
 E\left(\min_{\kappa\le i\le \kappa+k}(\gamma_i Z_i)- \min_{\kappa\le i\le \kappa + k+1}(\gamma_i Z_i)\right) + o(1), & 0\le k<d-\kappa.
\end{cases}
\]
\item[(ii)] $P(L_\kappa(s)\le k\mid X_\kappa > s)$
\[
=\begin{cases}
1, &k=d-\kappa\\
1- E\left(\min_{\kappa\le i\le \kappa+k+1}(\gamma_i Z_i)\right) + o(1),& 0\le k< d-\kappa.
\end{cases}
\]
\end{itemize}
\end{cor}

We, thus, obtain the limit distribution of the exceedance cluster length:
\begin{align*}
Q_\kappa([0,k])&:= \lim_{s\nearrow \omega^*} P(L_\kappa(s)\le k\mid X_\kappa > s)\\
&= \begin{cases}
1, &k=d-\kappa\\
1- E\left(\min_{\kappa\le i\le \kappa+k+1}(\gamma_i Z_i)\right),& 0\le k< d-\kappa.
\end{cases}
\end{align*}

Take, for example, the generator $\bfZ=2(U_1,\dots,U_d)$, where the $U_i$ are independent and uniformly on $(0,1)$ distributed rv. If, in addition, $\gamma_i=1$, $\kappa\le i\le d$, then we obtain
\[
Q_\kappa([0,k])= \begin{cases}
1,& k=d-\kappa\\
1-\frac 2{k+3},& 0\le k< d-\kappa.
\end{cases}
\]

Next we compute the asymptotic mean exceedance cluster length.

\begin{prop}\label{prop:sojourn_time}
Assume the conditions of Corollary \ref{cor:expansion_of_P(X_k_le s)} and let $\bfZ$ be a generator of the $D$-norm $\norm{\cdot}_D$. Then we have for $1\le \kappa\le d$
\begin{align*}
E\left(L_\kappa(s) \mid X_\kappa >s\right)
&=\begin{cases}
0,&\mbox{if }\kappa = d\\
\sum_{k=1}^{d-\kappa}s_\kappa(k) + o(1)&\mbox{else}
\end{cases}\\
&= \begin{cases}
0,&\mbox{if }\kappa = d\\
\sum_{k=1}^{d-\kappa} E\left(\min_{\kappa\le i\le \kappa+k}(\gamma_i Z_i)\right) + o(1) &\mbox{else}.
\end{cases}
\end{align*}
\end{prop}

\begin{proof}
Since $L_\kappa(s)$ attains only nonnegative values, we have for $\kappa<d$
\begin{align*}
E\left(L_\kappa(s) \mid X_\kappa >s\right)
&=\int_0^\infty P\left(L_\kappa(s)\ge t\mid X_\kappa >s\right)\,dt\\
&=\sum_{k=1}^{d-\kappa} P\left(L_\kappa(s)\ge k\mid X_\kappa >s\right)\\
&=\sum_{k=1}^{d-\kappa} P\left(X_\kappa>s,\dots,X_{\kappa+k}>s \mid  X_\kappa>s\right)\\
&=\sum_{k=1}^{d-\kappa}s_\kappa(k) + o(1).
\end{align*}
 \end{proof}

\begin{cor}\label{coro:mean_exceedance_cluster_length_for_L1-norm}
Under the conditions of the preceding result we have for $\kappa<d$, if $\gamma_k>0$, $1\le k\le d$´,
\[
\lim_{s\uparrow \omega^*}E(L_\kappa(s)\mid X_\kappa>s)=0
\]
if and only if $\norm{x\bfe_\kappa + y\bfe_{\kappa+1}}_D=
\norm{x\bfe_\kappa + y\bfe_{\kappa+1}}_1=x+y$, $x,y\ge 0$.
\end{cor}

\begin{proof}
Note that $s_\kappa(1)\ge\dots\ge s_\kappa(d-\kappa)$. We, thus, obtain from Proposition \ref{prop:sojourn_time}
\[
\lim_{s\uparrow \omega^*}E(L_\kappa(s)\mid X_\kappa>s)=0 \iff s_\kappa(1)=0.
\]
The assertion is now a consequence of Proposition 6.1 in \cite{falktichy10a}.
\end{proof}

Suppose in addition to the assumptions of Corollary
\ref{cor:expansion_of_P(X_k_le s)} that the components
$X_1,\dots,X_d$ of the rv $\bfX$ are exchangeable. Then we have
$\gamma_1=\dots=\gamma_d=1$, as well as $\norm{\sum_{i\in
T}\bfe_i}_D= \norm{\sum_{i=1}^{\abs{T}}\bfe_i}_D$ for any nonempty
subset $T\subset\set{1,\dots,d}$. As a consequence we obtain
\[
s_\kappa(k)=\sum_{j=1}^{k+1}(-1)^{j+1}\binom{k+1}{j} \norm{\sum_{i=1}^j\bfe_i}_D,\qquad 0\le k\le d-\kappa,
\]
and, thus, by rearranging sums,
\begin{align}\label{eqn:asymptotic_expectation_of_sojourn_time}
\lim_{s\nearrow} E\left(L_\kappa(s)\mid X_\kappa>s\right)&=\sum_{k=1}^{d-\kappa}s_\kappa(k)\nonumber\\
&=\sum_{j=1}^{d-\kappa+1} (-1)^{j+1}\norm{\sum_{i=1}^j\bfe_i}_D\sum_{k=\max(1,j-1)}^{d-\kappa}\binom{k+1}{j}\nonumber\\
&=-1+ \sum_{j=1}^{d-\kappa+1} (-1)^{j+1} \binom{d-\kappa+2}{j+1}\norm{\sum_{i=1}^j\bfe_i}_D,
\end{align}
where the final equality follows from the general equation $\sum_{r=n}^N\binom r n =\binom{N+1}{n+1}$.

\begin{exam}[Marshall-Olkin $D$-norm]\label{exam:Marshall-Olkin_D-norm}
The Marshall-Olkin $D$-norm is the convex combination of the maximum-norm and the $L_1$-norm:
\[
\norm{\bfx}_{\mathrm{MO}}=\vartheta\norm{\bfx}_1+(1-\vartheta)\norm{\bfx}_\infty,\qquad \bfx\in\R^d,\,\vartheta\in[0,1],
\]
see \cite[Example 4.3.4]{fahure10}. In this case we obtain from equation \eqref{eqn:asymptotic_expectation_of_sojourn_time}
\begin{equation*}
\lim_{s\nearrow} E\left(L_\kappa(s)\mid X_\kappa>s\right)=(1-\vartheta)(d-\kappa),
\end{equation*}
where we have used the general equation $\sum_{j=0}^m(-1)^j\binom m j=0$.

In the case $\vartheta=0$ of complete tail dependence of the
margins we, therefore, obtain $\lim_{s\nearrow}
E\left(L_\kappa(s)\mid X_\kappa>s\right)= d-\kappa,$ which is the
full possible length, whereas in the tail independence case
$\vartheta=1$ we obtain the shortest length $\lim_{s\nearrow}
E\left(L_\kappa(s)\mid X_\kappa>s\right)$ $=0$, which is in complete
accordance with Corollary
\ref{coro:mean_exceedance_cluster_length_for_L1-norm}.
\end{exam}

\end{document}